%% file: homotopy-low.tex
\documentclass[11pt]{article}
\usepackage[utf8]{inputenc}
\usepackage[T1]{fontenc}
\usepackage{graphicx}
\usepackage{longtable}
\usepackage{wrapfig}
\usepackage{rotating}
\usepackage[normalem]{ulem}
\usepackage{amsmath}
\usepackage{amssymb}
\usepackage{capt-of}
\usepackage{hyperref}
\usepackage{amsthm}
\theoremstyle{plain}
\newtheorem{theorem}{Theorem}[section]
\newtheorem{lemma}[theorem]{Lemma}
\newtheorem{corollary}[theorem]{Corollary}
\newtheorem*{MT}{Main Theorem}
\usepackage{tikz}
\author{Mauricio Islas-Gómez\\ email: \texttt{is294777@uaeh.edu.mx}\\ Rafael Villarroel-Flores\({}^{ * }\)\\ email: \texttt{rafaelv@uaeh.edu.mx}\\[12pt] Universidad Autónoma del Estado de Hidalgo\\ Carretera Pachuca-Tulancingo km. 4.5\\ Pachuca 42184 Hgo. México\\[12pt]  MSC: 05C76, 05E45\\ keywords: iterated clique graphs, topological combinatorics\\[12pt] \({}^{ * }\)Corresponding Author. Partially\\ supported by CONACYT, grant A1-S-45528.}
\date{\today}
\title{On the homotopy type of the iterated clique graphs of low degree}
\hypersetup{
 pdfauthor={Rafael},
 pdftitle={On the homotopy type of the iterated clique graphs of low degree},
 pdfkeywords={},
 pdfsubject={},
 pdfcreator={Emacs 28.1 (Org mode 9.5.2)}, 
 pdflang={English}}
\begin{document}

\maketitle
\begin{abstract}
To any simple graph \(G\), the clique graph operator \(K\) assigns the graph \(K(G)\) which is the intersection graph of the maximal complete subgraphs of \(G\). The iterated clique graphs are defined by \(K^{0}(G)=G\) and \(K^{n}(G)=K(K^{n-1}(G))\) for \(n\geq 1\). We associate topological concepts to graphs by means of the simplicial complex \(\mathrm{Cl}(G)\) of complete subgraphs of \(G\). Hence we say that the graphs \(G_{1}\) and \(G_{2}\) are homotopic whenever \(\mathrm{Cl}(G_{1})\) and \(\mathrm{Cl}(G_{2})\) are. A graph \(G\) such that \(K^{n}(G)\simeq G\) for all \(n\geq1\) is called \emph{\(K\)-homotopy permanent}. A graph is \emph{Helly} if the collection of maximal complete subgraphs of \(G\) has the Helly property. Let \(G\) be a Helly graph. Escalante (1973) proved that \(K(G)\) is Helly, and Prisner (1992) proved that \(G\simeq K(G)\), and so Helly graphs are \(K\)-homotopy permanent. We conjecture that if a graph \(G\) satisfies that \(K^{m}(G)\) is Helly for some \(m\geq1\), then \(G\) is \(K\)-homotopy permanent. If a connected graph has maximum degree at most four and is different from the octahedral graph, we say that it is a \emph{low degree graph}. 
It was recently proven that all low degree graphs \(G\) satisfy that \(K^{2}(G)\) is Helly. In this paper, we show that all low degree graphs have the homotopy type of a wedge or circumferences, and that they are \(K\)-homotopy permanent.
\end{abstract}

\section{Introduction}
\label{sec:orgde4c4da}

All graphs in this paper are finite and simple. Following \cite{MR0256911-djvu}, a \emph{clique} in a graph \(G\) is a maximal and complete subgraph of \(G\). The \emph{clique graph} of \(G\), denoted by \(K(G)\), is the intersection graph of its cliques. We define the sequence of \emph{iterated clique graphs} by \(K^{0}(G)=G\) and \(K^{n}(G)=K(K^{n-1}(G))\) for \(n\geq 1\). We usually identify a set of vertices of \(G\) with the subgraph of \(G\) that induces. In this way, if \(x\in G\) then \(G-x\) denotes the subgraph of \(G\) induced by its vertices different from \(x\). For \(x\in G\), we denote with \(N(x)\) the set of neighbors of \(x\) in \(G\), and with \(N[x]\) the set \(N(x)\cup\{x\}\). The maximum degree of a vertex in \(G\) is denoted by \(\Delta(G)\) and the cardinality of a set \(X\) by \(|X|\).

Denote by \(\mathrm{Cl}(G)\) the simplicial complex where the simplices are the complete subgraphs of \(G\).  In this way, we may assign topological concepts to graphs, and we say that the graphs \(G_{1}\), \(G_{2}\) are \emph{homotopic} (denoted as \(G_{1}\simeq G_{2}\)), whenever \(\mathrm{Cl}(G_{1})\), \(\mathrm{Cl}(G_{2})\) are homotopy equivalent. As in \cite{MR3129782}, we will say that a graph \(G\) is \emph{homotopy \(K\)-invariant} if \(G\simeq K(G)\), and that the graph \(G\) is homotopy \emph{\(K\)-permanent} if \(G\simeq K^{n}(G)\) for all \(n\geq1\).

The graph of the octahedron \(O_{3}\) can be defined as the complement of the disjoint union of 3 edges. This is a \(4\)-regular graph. If a connected graph \(G\) has \(\Delta(G)\leq 4\) and \(G\) is not the octahedron, we will say that the graph \(G\) is a \emph{low degree graph}.

The purpose of this paper is to prove the following:

\begin{MT}
If \(G\) is a low degree graph, then:
\begin{itemize}
\item \(G\) is homotopy equivalent to a wedge of circumferences.
\item \(G\) is homotopy \(K\)-permanent.
\end{itemize}
\end{MT}

A collection of sets is \emph{intersecting} if any two members of the collection has nonempty intersection. A collection of sets has the \emph{Helly property} if any intersecting subcollection has nonempty intersection. A graph \(G\) is called \emph{Helly} if the collection of its cliques have the Helly property. The consideration of homotopy \(K\)-invariant graphs started in \cite{Pri92}, where it is proven that Helly graphs are homotopy \(K\)-invariant. Prisner's result was generalized in \cite{LARRION2001}, and then generalized further in \cite{LPV08a}. Since it was already known from \cite{MR0329947} that \(K(G)\) is Helly if \(G\) is Helly, then one actually obtains that Helly graphs are homotopy \(K\)-permanent.

A vertex \(x\in G\) is \emph{dominated} by \(y\in G\) if \(N[x]\subseteq N[y]\). The graph \(G\) is \emph{dismantlable} if it is the graph of one vertex, or if there is \(x\in G\) dominated by \(y\ne x\) such that \(G-x\) is dismantlable. Removing dominated vertices does not alter the homotopy type of a graph (\cite{Pri92}), and so dismantlable graphs are contractible. If \(G\) is dismantlable, then \(K(G)\) is dismantlable (\cite{BP91}), hence dismantlable graphs are homotopy \(K\)-permanent.

We say that the graph \(G\) is \emph{convergent} if the sequence of iterated clique graphs has, up to isomorphism, a finite number of graphs. If \(G\) is not convergent we say that \(G\) is \emph{divergent}. The graph of the octahedron was the first known example of a divergent graph, given by Neumann-Lara in \cite{zbMATH03641500}. In fact, for each \(n\geq 1\), we have that the octahedral graph \(O_{n}\) can be defined as the complement of the disjoint union of \(n\) edges, that is, \(O_{n}=\overline{nK_{2}}\). The result of Neumann-Lara is that \(K(O_{n})=O_{2^{n-1}}\). Thus, \(O_{n}\) is divergent for \(n\geq 3\). Since \(\mathrm{Cl}(O_{n})\) is homeomorphic to the sphere \(S^{n-1}\), we have that \(O_{3}\) is not homotopy \(K\)-invariant and our Main Theorem will show that \(O_{3}\) is the only connected graph \(G\) with \(\Delta(G)\leq 4\) that is not homotopy \(K\)-invariant. On the other hand, all Helly graphs (\cite{MR0329947}) and all dismantlable graphs (\cite{Pri92}) are convergent. There are also examples of divergent graphs that are homotopy \(K\)-permanent, see \cite{equivariantc}.

It could happen that a non-Helly graph \(G\) satisfies that \(K^{m}(G)\) is Helly for some \(m\geq 1\). All such graphs seem to be also homotopy \(K\)-permanent. In fact, the problem about the contractibility of graphs such that \(K^{m}(G)\) is a point for some \(m\geq 1\) was considered in \cite{contractibility}, where it was proven that if \(K(G)\) has a universal vertex (which includes all graphs \(G\) where \(K^{2}(G)\) is a point) then \(G\) is contractible.

In this paper we study the homotopy type of the iterated clique graphs of low degree graphs. The family of low degree graphs includes an infinite number of graphs for which both \(G\) and \(K(G)\) are neither Helly nor dismantlable. However, it was recently proved (\cite{2022-on-the-clique-behavior-of-graphs-of-low-degree}) that \(K^{2}(G)\) is always Helly for a low degree graph \(G\), and so the known results imply that \(K^{n}(G)\simeq K^{2}(G)\) for all \(n\geq 2\). Here, we show that \(G\simeq K(G)\simeq K^{2}(G)\), so any low degree graph \(G\) is actually homotopy \(K\)-permanent.

Given a graph \(G\), let \(\mathcal{C}\) be the collection of subcomplexes of \(\mathrm{Cl}(G)\) indexed by the cliques of \(G\), where the subcomplex corresponding to \(q\in K(G)\) consists of the complete subgraphs contained in \(q\). This cover satisfies the hypothesis of the Nerve Theorem (Theorem 10.6 from \cite{1995-bjorner-topological-methods}), and so, following the notation of \cite{2016-adamaszek-adams-frick-peterson-previte-johnson-nerve-complexes-of-circular-arcs}, we have that \(\mathcal{N}(\mathcal{C})\simeq \mathrm{Cl}(G)\) (In the notation of \cite{2012-barmak-minian-strong-homotopy-types-nerves-and-collapses}, we have \(\mathcal{N}(\mathcal{C})=\mathcal{N}(\mathrm{Cl}(G))\)). On the other hand, \(\overline{\mathcal{N}}(\mathcal{C})\) is \(\mathrm{Cl}(K(G))\). We have that \(\mathcal{N}(\mathcal{C})=\overline{\mathcal{N}}(\mathcal{C})\) precisely when \(G\) is Helly, and this is another way of looking at Prisner's result. Therefore, our result can be interpreted as giving a sufficient condition, which applies even when \(G\) is not Helly, under which the complex \(\mathcal{N}(\mathcal{C})\) (the Čech complex of the cover) is homotopy equivalent to the complex \(\overline{\mathcal{N}}(\mathcal{C})\) (the Vietoris-Rips complex of the cover).

\section{Preliminaries}
\label{sec:orgb3bae06}

Given \(x\in G\), we define the \emph{star} of \(x\), denoted by \(x^{*}\), as the set of cliques of \(G\) that contain \(x\), that is, \(x^{ * }=\{q\in K(G)\mid x\in q\}\). This is a complete subgraph of \(K(G)\). If \(x^{ *} \in K^{2}(G)\), we say that \(x\) is a \emph{normal vertex}. If \(Q\in K^{2}(G)\) is such that \(\cap Q=\emptyset\), we say that \(Q\) is a \emph{necktie} of \(G\). Thus, vertices of \(K^{2}(G)\) are partitioned in stars and neckties of \(G\), and \(G\) is Helly precisely when \(G\) has no neckties.
We say that a triangle \(T\) in the low degree graph \(G\) is \emph{internal} if it is a clique, and every edge of \(T\) is contained in a clique different from \(T\). Given an inner triangle \(T\) in \(G\), we define \(Q_{T}=\{q\in K(G)\mid |q\cap T|\geq 2\}\). This is always a complete in \(K(G)\).

We collect the following facts about low degree graphs from \cite{2022-on-the-clique-behavior-of-graphs-of-low-degree}:

\begin{theorem}
\label{facts-from-previous-paper}
Let \(G\) be a low degree graph. Then:
\begin{enumerate}
\item \label{facts-necktie-is-centered} If \(Q\) is a necktie in \(G\), then \(Q\) consists only of triangles and there is an internal triangle \(T\) such that \(Q=Q_{T}=\{T,T_{1},T_{2},T_{3}\}\), where \(\{T\cap T_{i}\mid i\in\{1,2,3\}\}\) are the three edges of \(T\). Conversely, if \(T\) is an internal triangle, then \(Q_{T}\) is a necktie.
\item \label{facts-intersecting-neckties} If \(T_{1}\) and \(T_{2}\) are internal triangles such that \(Q_{T_{1}}\sim Q_{T_{2}}\) in \(K^{2}(G)\), then \(T_{1}\cap T_{2}\ne\emptyset\). If \(|T_{1}\cap T_{2}|=1\), then the graph of Figure \ref{intersection-in-one} is a subgraph of \(G\).  If \(|T_{1}\cap T_{2}|=2\), then the graph of Figure \ref{intersection-in-two} is a subgraph of \(G\).
\item \label{second-iterated-is-helly} \(K^{2}(G)\) is Helly.
\end{enumerate}
\end{theorem}

For the internal triangle \(T\), we say that \(T\) is the \emph{center} of the necktie \(Q_{T}\). Any other triangle in \(Q_{T}\) is called an \emph{ear} of \(Q_{T}\).

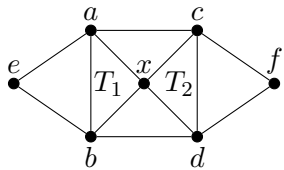
\begin{figure}
\centering
\input{intersection-in-one.tikz}
\caption{Case \(|T_{1}\cap T_{2}|=1\) \label{intersection-in-one}}
\end{figure}

\begin{figure}
\centering
\input{intersection-in-two.tikz}
\caption{Case \(|T_{1}\cap T_{2}|=2\) \label{intersection-in-two}}
\end{figure}
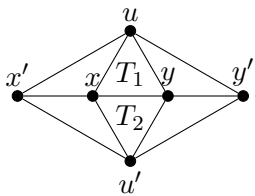

\begin{theorem}
\label{prisner-removing-dominated}
(Proposition 3.2 from \cite{Pri92}) If \(x\in G\) is dominated by a vertex \(y\ne x\), then \(G-x\simeq G\).
\end{theorem}

Given an edge \(e=\{x,y\}\) of \(G\), let \(N[e]=N[x]\cap N[y]\). The following result allows us to remove edges from a graph without altering its homotopy type:

\begin{theorem}
\label{removing-edge}
(Proposition 2.3 from \cite{equivariantc}, see also Lemma 1.6 from \cite{2010-boulet-fieux-jouve-simplicial-simple-homotopy-of-flag-complexes-in-terms-of-graphs}) Let \(G\) be a graph such that \(e\) is an edge properly contained in \(N[e]\), and \(N[e]\) is complete. Then \(G\simeq G-e\).
\end{theorem}

If \(H\) and \(G\) are graphs such that \(G\) has no induced subgraph isomorphic to \(H\), we say that \(G\) is \emph{\(H\)-free}. It is immediate that any low degree graph is \(O_{3}\)-free. 

\begin{theorem}
\label{triangle-in-unique-clique}
(Theorem 4.2 from \cite{chessboard-graphs}) Let \(G\) be an \(O_{3}\)-free graph such that every triangle in \(G\) is contained in a unique clique. Then \(G\) is homotopy \(K\)-invariant.
\end{theorem}

Following \cite{MR2059528}, we write \(G\xrightarrow{{\#}} H\) if \(H\) is isomorphic to a subgraph \(H_{0}\) of \(G\) such that every vertex in \(G\) not in \(H_{0}\) is dominated by a vertex in \(H_{0}\). We have then that \(G\xrightarrow{{\#}}H\) implies \(G\simeq H\). If \(x,y\in G\) are such that \(N[x]=N[y]\) (so that they dominate each other), we say that \(x,y\) are \emph{twins}. 

\begin{theorem}
\label{hash-arrow-in-cliques}
(Theorem 3 from \cite{MR2059528}) If  \(G\xrightarrow{{\#}} H\), then  \(K(G)\xrightarrow{{\#}} K(H)\). 
\end{theorem}

If \(q_{1},q_{2},\ldots,q_{n}\) are distinct cliques in a graph \(G\) such that none of them is a subset of the union of the others, then for each \(i=1,\ldots,n\) there is a vertex \(x_{i}\in q_{i}\) that is not in any of the other cliques. In this case we say that \((x_{1},\ldots,x_{n})\) is a \emph{selection from \((q_{1},q_{2},\ldots,q_{n})\)}. 

\begin{lemma}
\label{clique-in-union}
Let \(q_{1},q_{2},\ldots,q_{n}\) be distinct cliques in a graph \(G\) such that none of them is a subset of the union of the others. Let \(q\) be a clique in \(G\) different from any of \(q_{1},q_{2},\ldots,q_{n}\) such that \(q\subseteq q_{1}\cup\cdots\cup q_{n}\) but \(q\) is not contained in the union of \(n-1\) cliques from \(q_{1},q_{2},\ldots,q_{n}\). Then there is a selection from \((q_{1},q_{2},\ldots,q_{n})\) that consists of elements from \(q\).
\end{lemma}

\begin{proof}
For each \(i\) with \(1\leq i\leq n\), take \(x_{i}\in q-(q_{1}\cup\cdots\cup \hat{q}_{i}\cup\cdots\cup q_{n})\). 
\end{proof}

\begin{lemma}
\label{different-cliques-subgraph}
Let \(H\) be an induced subgraph of \(G\). If \(c\) is a complete in \(H\) that is contained in two different cliques of \(H\), then \(c\) is contained in two different cliques of \(G\).
\end{lemma}

\begin{proof}
Suppose that the only clique of \(G\) that contains \(c\) is \(q\). Let \(c\subseteq q_{1}\cap q_{2}\), where \(q_{1}, q_{2}\) are different cliques of \(H\). Since \(q_{1},q_{2}\) are completes of \(G\), they can be extended to cliques of \(G\). But our assumption implies that both \(q_{1},q_{2}\) are contained in \(q\). Since the subgraph \(H\) is induced, we must have that \(q_{1}\cup q_{2}\) is a complete in \(H\), and so \(q_{1}=q_{2}\), which is a contradiction. 
\end{proof}

\begin{lemma}
\label{clique-touching-three-triangles}
Let \(G\) be a graph with \(\Delta(G)\leq 4\). Let \(x,y_{1},y_{2},z_{1},z_{2}\in G\) such that \(q=\{x,y_{1},z_{1}\}\), \(q_{1}=\{x,y_{1},y_{2}\}\), \(q_{2}=\{x,z_{1},z_{2}\}\) are cliques. Let \(q'\in K(G)\) such that \(x\in q'\) and \(q'\not\in\{q,q_{1},q_{2}\}\). Then \(q'=\{x,y_{2},z_{2}\}\) and so \(y_{2}\sim z_{2}\).
\end{lemma}

\begin{proof}
We have \(q'\subseteq N[x]=\{x,y_{1},z_{1},y_{2},z_{2}\}\). We cannot have \(y_{1}\in q'\) because \(\{x,y_{1}\}\) cannot be extended to a clique different from \(q,q_{1}\). Similarly, we have \(z_{1}\not\in q'\). Hence \(q'=\{x,y_{2},z_{2}\}\).
\end{proof}

\begin{lemma}
\label{facts-intersection-with-internal}
Let \(G\) be a graph with \(\Delta(G)\leq 4\). If \(T\) is an internal triangle, and \(T'\) is any triangle such that \(T\cap T'=\{x\}\), then \(N(x)\) induces a \(4\)-cycle.
\end{lemma}

\begin{proof}
Suppose \(T=\{x,y,z\}\), \(T'=\{x,y',z'\}\). Since \(T\) is internal, there is \(q\in K(G)\) such that \(q\ne T\) and \(q\cap T = \{x,y\}\). Then \(q\subseteq N[x] = \{x,y,z,y',z'\}\). Since \(|q|\geq 3\) and \(z\not\in q\), we have that one of \(y',z'\) is in \(q\). Without loss, assume \(y'\in q\) and so \(y\sim y'\). Now, considering \(q'\ne T\) such that \(q'\cap T = \{x,z\}\) we obtain that \(z\sim z'\).
\end{proof}

\begin{lemma}
\label{all-triangles-internal}
Let \(G\) be a graph with \(\Delta(G)\leq 4\), with an internal triangle \(T\) such that all the ears of \(Q_{T}\) are internal triangles. Then \(G\cong O_{3}\).
\end{lemma}

\begin{proof}
Let \(T=\{x,y,z\}\) be such triangle. Let \(T_{1}=\{x,y,z'\}\) be a triangle in \(G\) with \(T_{1}\ne T\), similarly, let \(T_{2}=\{x,z,y'\}\) be a triangle with \(T_{2}\ne T\), and \(T_{3}=\{y,z,x'\}\) with \(T_{3}\ne T\). Then \(T_{1}, T_{2}, T_{3}\) are internal. Since \(T_{1}\cap T_{2}=\{x\}\), from Lemma \ref{facts-intersection-with-internal} we get that \(N(x)\) is a 4-cycle, hence \(y'\sim z'\).  We can obtain similarly that \(x'\sim y'\) and \(y'\sim z'\). It follows that \(G\cong O_{3}\).
\end{proof}

\begin{lemma}
\label{twin-internal-triangles}
Let \(G\) be a graph with \(\Delta(G)\leq 4\), with two internal triangles \(T_{1},T_{2}\) sharing an edge such that as vertices of \(K(G)\) they are not twins. Then \(G\cong O_{3}\).
\end{lemma}

\begin{proof}
Suppose \(T_{1}=\{x,y,u\}\), \(T_{2}=\{x,y,u'\}\). Since \(T_{1}\) is internal, there is \(x'\ne y\) with \(x'\sim x\) and \(x'\sim u\), and there is \(y'\ne x\) with \(y'\sim u\) and \(y'\sim y\). Since \(T_{1}\) is a clique, we have that \(x'\ne y'\). Then \(T_{2}\cap\{x,x',u\}=\{x\}\). By Lemma \ref{facts-intersection-with-internal}, \(x'\sim u'\). Similarly, \(T_{2}\cap \{u,y,y'\}=\{y\}\), hence \(y'\sim u'\). Suppose that \(T_{2}\) does not dominate \(T_{1}\) in \(K(G)\). Then there is \(q\in K(G)\) with \(q\cap T_{1}\ne\emptyset\) and \(q\cap T_{2}=\emptyset\). Hence \(q\cap T_{1}=\{u\}\). Then Lemma \ref{clique-touching-three-triangles} applies to prove \(x'\sim y'\) and so \(G\cong O_{3}\).
\end{proof}

\section{The homotopy type of a low degree graph}
\label{sec:org4e3ffea}

The \emph{clique number} \(\omega(G)\) is the maximum order of a clique in \(G\).

\begin{theorem}
\label{wedge-of-circumferences}
If \(G\) is a low degree graph, then \(G\) is homotopy equivalent to a wedge of circumferences.
\end{theorem}

\begin{proof}
Let \(G\) be a low degree graph. It will be enough to show that \(G\) is homotopy equivalent to a triangleless graph (see Example 0.7 from \cite{MR1867354}). Since the only low degree graph with \(\omega(G)\geq 5\) is the complete graph \(K_{5}\) (and \(\mathrm{Cl}(K_{5})\) is contractible, and so homotopic to an empty wedge of circumferences), we can assume that \(\omega(G)\leq 4\). We prove first that there is a graph \(G'\) such that \(\omega(G')\leq 3\) and \(G'\simeq G\). Suppose then that \(G\) contains a clique \(q\) of four vertices. If a vertex in \(q\) has no neighbor outside of \(q\), then such vertex is dominated by any of its neighbors, and so by Theorem \ref{prisner-removing-dominated} it can be removed without altering the homotopy type of \(G\). Suppose then that that all vertices of \(q\) have a neighbor outside of \(q\), but that \(x,y\in q\) share \(w\not\in q\) as a neighbor. Given that the degree of vertices in \(G\) is at most four, we must have that \(N[x]=N[y]=q\cup\{x\}\), and so \(x\) is dominated by \(y\) and can be removed without affecting the homotopy type of \(G\). Suppose then that all vertices of \(q\) have neighbors outside \(q\), and such four vertices are different. In that case, for any edge \(e\in q\) we have that \(N[e]=q\), and by Theorem \ref{removing-edge}, we may remove \(e\) and get a homotopy equivalent graph. In this way, we obtain a graph \(G'\simeq G\) that has no cliques of four vertices, and is low degree.

Without loss, we may assume then that the low degree graph \(G\) has \(\omega(G)\leq 3\). There must be a triangle \(T\) in \(G\) that is not internal, otherwise by Lemma \ref{all-triangles-internal}, \(G\) would be \(O_{3}\). Then one edge of \(T\) is such that the only clique that contains such edge is \(T\). It follows that we may apply Theorem \ref{removing-edge} to remove the edge and get a homotopy equivalent graph. We can thus remove inductively all triangles of \(G\), to arrive at the desired triangleless graph \(G'\simeq G\).
\end{proof}

We obtain thus the first part of the Main Theorem.

\section{A low degree graph is homotopy \(K\)-invariant}
\label{sec:org8f8ae12}

In the next two theorems, we denote a triangle with a small \(t\), since we are not claiming that the triangle is a clique.

\begin{theorem}
\label{g-triangle-in-two-cliques}
Let \(G\) be a low degree graph, and let \(t=\{x_{1},x_{2},x_{3}\}\) be a triangle that is contained in two different cliques of \(G\). Then the vertices \(x_{1},x_{2},x_{3}\) are twins.
\end{theorem}

\begin{proof}
Suppose that \(t\subseteq q_{1}\cap q_{2}\), where \(q_{1}, q_{2}\) are two different cliques of \(G\). Then there are \(x'_{1},x'_{2}\in G\) such that \(q_{i}=t\cup\{x'_{i}\}\) for \(i=1,2\). Then we must have \(N[x_{1}]=N[x_{2}]=N[x_{3}]=t\cup\{x'_{1},x'_{2}\}\), and so our claim follows.
\end{proof}

\begin{theorem}
\label{low-degree-homotopy-invariant}
Let \(G\) be a low degree graph. Then \(K(G)\simeq G\).
\end{theorem}

\begin{proof}
Let \(H\) be the graph obtained from \(G\) by removing vertices as follows: For each triangle \(t\) contained in two different cliques of \(G\), remove exactly two vertices of \(t\). By Lemma \ref{g-triangle-in-two-cliques}, we have that \(G\xrightarrow{{\#}} H\). We claim that \(H\) satisfies the hypothesis of Theorem \ref{triangle-in-unique-clique}. First, since \(H\) is an induced subgraph of \(G\), we obtain that \(H\) is \(O_{3}\)-free. Now, suppose that \(t=\{x_{1},x_{2},x_{3}\}\) is a triangle in \(H\) contained in the cliques \(q_{1},q_{2}\) of \(H\), with \(q_{1}\ne q_{2}\). Then by Lemma \ref{different-cliques-subgraph}, we have that \(t\) is a triangle in \(G\) contained in different cliques of \(G\), and so, \(x_{1},x_{2},x_{3}\) are twin vertices in \(G\). By our construction of the graph \(H\), it is not possible that \(t\subseteq H\). Hence Theorem \ref{triangle-in-unique-clique} applies, and so \(K(H)\simeq H\). By Theorem \ref{hash-arrow-in-cliques}, we get that \(K(G)\xrightarrow{{\#}} K(H)\), and so \(K(G)\simeq K(H)\simeq H\simeq G\).
\end{proof}

\section{The clique graph of a low degree graph is homotopy \(K\)-invariant}
\label{sec:org990e880}

For \(G\) a low degree graph, in this section we will prove that there is a graph obtained from \(K(G)\) removing dominated vertices that satisfies the hypothesis of Theorem \ref{triangle-in-unique-clique}. We will consider the implications of having a triangle in \(K(G)\) contained in two cliques of \(K(G)\).

\begin{lemma}
\label{triangle-in-two-stars}
Let \(G\) be a low degree graph. Let \(\{q_{1},q_{2},q_{3}\}\) be a triangle in \(K(G)\) that is contained in \(x^{ * }\cap y^{ * }\) for \(x,y\in G\). Then \(x^{ * }= y^{ * }\).
\end{lemma}

\begin{proof}
In this case, we have that \(\{x,y\}\in q_{1}\cap q_{2}\cap q_{3}\). Suppose we had \(q_{1}\subseteq q_{2}\cup q_{3}\). By Lemma \ref{clique-in-union}, there is a selection \((z_{2}, z_{3})\) from \((q_{2},q_{3})\) in \(q_{1}\), that is, there are vertices \(z_{2}\in q_{1}\cap q_{2}\) with \(z_{2}\not\in q_{3}\) and \(z_{3}\in q_{1}\cap q_{3}\) with \(z_{3}\not\in q_{2}\). Let \(w\in q_{2}\) with \(w\not\in q_{1}\). Then \(x,y,z_{2},z_{3},w\) are five distinct vertices, hence \(N(x)=\{y,z_{2},z_{3},w\}\). Let \(w'\in q_{3}\) with \(w'\not\in q_{1}\). We have that \(w'\in N(x)\), but since \(w'\not\in q_{1}\) we must have \(w'=w\). It would follow then that \(\{x,y,z_{2},z_{3},w=w'\}\) is a complete of five vertices in \(G\), which is a contradiction. It follows that among \(q_{1},q_{2},q_{3}\), none of them is contained in the union of the other two. Let \((x_{1},x_{2},x_{3})\) be a selection from \((q_{1},q_{2},q_{3})\). Then \(\{x,y,x_{1},x_{2},x_{3}\}\) are five distinct vertices, and so, this set is equal to \(N[x]\). It follows also that \(q_{1} = \{x,y,x_{1}\}\), since if there was a vertex \(x'\in q_{1}-\{x,y,x_{1}\}\), then \(x'\in N[x]\) but neither \(x_{2}\) nor \(x_{3}\) are elements of \(q_{1}\). We obtain also that \(q_{2} = \{x,y,x_{2}\}\) and that \(q_{3} = \{x,y,x_{3}\}\).

We claim that \(x^{ * } = \{q_{1},q_{2},q_{3}\}\). Suppose not, and let \(q\in x^{ * }\) with \(q\not\in\{q_{1},q_{2},q_{3}\}\). By Lemma \ref{clique-in-union}, since \(q\subseteq N[x]=q_{1}\cup q_{2}\cup q_{3}\), we would obtain that \(\{x_{1},x_{2},x_{3}\}\) is a complete, which is a contradiction. Hence \(x^{ * }=\{q_{1},q_{2},q_{3}\}\). Similarly, we can prove \(\{q_{1},q_{2},q_{3}\}\) is also equal to \(y ^{*}\).
\end{proof}

\begin{lemma}
\label{triangle-in-two-neckties}
Let \(G\) be a low degree graph.  Let \(\{q_{1},q_{2},q_{3}\}\) be a triangle in \(K(G)\) that is contained in \(Q_{T_{1}}\cap Q_{T_{2}}\) where \(T_{1},T_{2}\) are internal triangles of \(G\). Then \(T_{1}=T_{2}\).
\end{lemma}

\begin{proof}
Assume that the hypothesis of the lemma are satisfied with \(T_{1}\ne T_{2}\). Then \(Q_{T_{1}}\sim Q_{T_{2}}\) in \(K^{2}(G)\), and so \(T_{1}\cap T_{2}\ne\emptyset\) by Theorem \ref{facts-from-previous-paper}. If \(|T_{1}\cap T_{2}|=1\), we are in the situation of Figure \ref{intersection-in-one}. However, in that case \(Q_{T_{1}}\cap Q_{T_{2}}\) consists only of the triangles \(\{x,a,c\},\{x,b,d\}\). If \(|T_{1}\cap T_{2}|=2\), we are in the situation of Figure \ref{intersection-in-two}, and in that case, \(Q_{T_{1}}\cap Q_{T_{2}}=\{T_{1},T_{2}\}\). Hence, we must have \(T_{1}=T_{2}\).
\end{proof}

\begin{lemma}
\label{kg-no-induced-octahedra}
If \(G\) is a low degree graph, then \(K(G)\) is \(O_{3}\)-free.
\end{lemma}

\begin{proof}
Let \(H\subseteq K(G)\) with \(H=\{q_{1},q_{2},q_{3},q'_{1},q'_{2},q'_{3}\}\) such that its complement \(\overline{H}\) consists exactly of the three edges \(\{q_{i},q'_{i}\}\) for \(i=1,2,3\). Then \(H\cong O_{3}\). As a first case, suppose that the cliques corresponding to the vertices of each of the eight faces of the octahedron have non-empty intersection. Then, in \(G\) we would have vertices: \(x_{123}\in q_{1}\cap q_{2}\cap q_{3}\), \(x_{123'}\in q_{1}\cap q_{2}\cap q'_{3},\ldots,\) \(x_{1'2'3'}\in q'_{1}\cap q'_{2}\cap q'_{3}\).  But then \(x_{123}\) would have six neighbors among these eight vertices, which is impossible. Hence, we may assume without loss that \(q_{1}\cap q_{2}\cap q_{3}=\emptyset\). Let \(x_{12}\in q_{1}\cap q_{2}\), \(x_{13}\in q_{1}\cap q_{3}\) and \(x_{23}\in q_{2}\cap q_{3}\). It is not possible that \(q_{1}\subseteq q_{2}\cup q_{3}\) because otherwise \(q_{1}\cup\{x_{23}\}\) would be a complete graph properly containing a clique. Let \((y_{1},y_{2},y_{3})\) be a selection from \((q_{1},q_{2},q_{3})\). Then \(N(x_{12})=\{y_{1},y_{2},x_{13},x_{23}\}\), and so, we have that \(x_{12}\not\sim y_{3}\). Similarly, \(x_{13}\not\sim y_{2}\) and \(x_{23}\not\sim y_{1}\). It also follows that \(T\) is a clique, and an internal triangle. Then \(\{q_{1},q_{2},q_{3}\}\) are the ears of the necktie \(Q_{T}\), and so they must be triangles: \(q_{1}=\{x_{12},x_{13},y_{1}\}\), \(q_{2}=\{x_{12},x_{23},y_{2}\}\) and \(q_{3}=\{x_{13},x_{23},y_{3}\}\).

On the other hand, \(q_{1}'\cap q_{2}\ne \emptyset\) and \(q_{1}'\cap q_{3}\ne\emptyset\). If \(x_{23}\in q_{1}'\), then by Lemma \ref{clique-touching-three-triangles}, we have that \(y_{2}\sim y_{3}\). If \(x_{23}\not\in q_{1}'\), then \(q_{1}'\cap q_{2} = \{y_{2}\}\) and \(q_{1}'\cap q_{3} = \{y_{3}\}\), and this also means \(y_{2}\sim y_{3}\). Similarly, our hypothesis on \(q_{2}'\) imply that \(y_{1}\sim y_{3}\) and our hypothesis on \(q_{3'}\) imply that \(y_{1}\sim y_{2}\). But then \(\{x_{12},x_{13},x_{23},y_{1},y_{2},y_{3}\}\) induce an octahedron in \(G\). This contradiction proves our claim.
\end{proof}

\begin{lemma}
\label{dominated-vertices-in-kg}
Let \(G\) be a low degree graph. Then:
\begin{enumerate}
\item If \(T\) is an internal triangle that does not share an edge with another internal triangle, then \(T\in K(G)\) is dominated in \(K(G)\) by one of its ears.
\item If \(T_{1}\) is an internal triangle that shares an edge with the internal triangle \(T_{2}\ne T_{1}\), then \(T_{1},T_{2}\) are twin vertices in \(K(G)\).
\end{enumerate}
\end{lemma}

\begin{proof}
For the first part, suppose that \(T\) is an internal triangle where \(Q_{T}\) has ears \(q_{1},q_{2},q_{3}\). For \(i=1,2,3\) let \(x_{i}\in q_{i}-T\). Suppose that \(q_{1}\) does not dominate \(T\). Then there must be a clique \(q\) such that \(q\cap T\ne \emptyset\) but \(q\cap q_{1}=\emptyset\). Suppose \(q_{2}\cap q_{3}=\{x_{1}'\}\). Then \(q\subseteq N[x_{1}']\) and Lemma \ref{clique-touching-three-triangles} would imply that \(x_{2}\sim x_{3}\). Similarly, if \(q_{2}\) does not dominate \(T\) we would have \(x_{1}\sim x_{3}\). But this would imply that \(q_{3}\) is an internal triangle, against our hypothesis on \(T\).

The second statement is immediate from Lemma \ref{twin-internal-triangles}.
\end{proof}

\begin{theorem}
\label{main-theorem}
If \(G\) is a low degree graph, then \(K(G)\) is homotopy \(K\)-invariant.
\end{theorem}

\begin{proof}
Let \(H\) be the subgraph obtained from \(K(G)\) removing vertices that are internal triangles of \(G\) according to the following rules:
\begin{enumerate}
\item Remove internal triangles that share no edge with another internal triangle.
\item \label{one-twin-removed} If an internal triangle shares an edge with another internal triangle, remove exactly one of them.
\end{enumerate}

In this way, by Lemma \ref{dominated-vertices-in-kg}, we obtain a graph \(H\) such that \(K(G)\xrightarrow{{\#}} H\) (so that \(K(G)\simeq H\)). We will prove that \(H\) satisfies the hypothesis of Theorem \ref{triangle-in-unique-clique}. First, since \(H\) is an induced subgraph of \(K(G)\), by Lemma \ref{kg-no-induced-octahedra} we obtain that \(H\) is \(O_{3}\)-free.

Let \(\{q_{1},q_{2},q_{3}\}\) be a triangle in \(H\) contained in more than one clique of \(H\). By Lemma \ref{different-cliques-subgraph}, this triangle is contained in more than one clique of \(K(G)\). By Lemma \ref{triangle-in-two-stars} and Lemma \ref{triangle-in-two-neckties} we have that such cliques of \(K(G)\) are one star and one necktie. Let \(x\in G\) be a normal vertex and \(T\) an internal triangle in \(G\) be such that \(\{q_{1},q_{2},q_{3}\}\subseteq x^{ * }\cap Q_{T}\). Then, Theorem \ref{facts-from-previous-paper}, item \ref{facts-necktie-is-centered}, gives us that \(Q_{T}\) consists of four triangles. The triangle \(T\) must be an element of the set \(\{q_{1},q_{2},q_{3}\}\), in order for \(\{q_{1},q_{2},q_{3}\}\) to be contained in a star. Without loss, assume that \(T=q_{1}\), and that we have the situation of Figure \ref{diagram-hch}, so that \(q_{1}=\{x,y,z\}\), and the ears of \(Q_{T}\) are \(q_{2}, q_{3}, q_{4}\).

\begin{figure}
\centering
\input{diagram-hch.tikz}
\caption{Dashed lines indicate not adjacent vertices \label{diagram-hch}}
\end{figure}
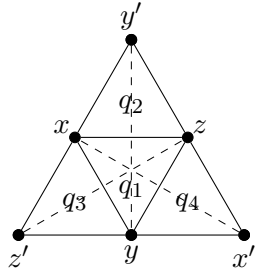

Since \(x\) is a normal vertex, there must be a clique \(q'\) such that \(x\in q'\) and \(q'\not\in\{q_{1},q_{2},q_{3}\}\). Then Lemma \ref{clique-touching-three-triangles} implies that \(y'\sim z'\). Now, the facts that \(q_{1}\in H\) and that \(q_{1}\) is internal imply that one of \(q_{2},q_{3},q_{4}\) must be internal (otherwise, \(q_{1}\) would have been removed and would not be in \(H\)). But if \(q_{2}\) were internal, then Lemma \ref{dominated-vertices-in-kg} implies that \(q_{1}\) and \(q_{2}\) are twins, and our removal rules would imply that one of \(q_{1},q_{2}\) should not be present in \(H\). Similarly, \(q_{3}\) cannot be internal. The only possibility is that \(q_{4}\) is internal and was removed from \(G\) and so \(q_{4}\not\in H\). But if \(q_{4}\) is internal, then Lemma \ref{facts-intersection-with-internal} applied to \(q_{4}\) and \(q_{2}\) implies that \(N(z)\) is a 4-cycle, and so \(x'\sim y'\). The same Lemma applied to \(q_{4},q_{3}\) implies that \(x'\sim z'\). So, we would obtain that \(G\cong O_{3}\). This contradiction means that \(H\) satisfies the hypothesis of Theorem \ref{triangle-in-unique-clique}, and so \(K(H)\simeq H\).
From \(K(G)\xrightarrow{{\#}} H\), we obtain \(K^{2}(G)\xrightarrow{{\#}} K(H)\simeq H\simeq K(G)\). This means that \(K^{2}(G)\simeq K(G)\).
\end{proof}

\section{Conclusion}
\label{sec:org77aba30}

If \(G\) is a low degree graph, from Theorem \ref{low-degree-homotopy-invariant} we obtained that \(G\simeq K(G)\), and from Theorem \ref{main-theorem}, we obtained \(K(G)\simeq K^{2}(G)\). Since \(K^{2}(G)\) is Helly by item \ref{second-iterated-is-helly} of Theorem \ref{facts-from-previous-paper}, we have \(K^{n}(G)\simeq K^{2}(G)\) for \(n\geq 2\). This proves that \(G\) is \(K\)-homotopy permanent, as claimed by the second part of the Main Theorem.

We conclude with the following consequences of the Main Theorem:

\begin{corollary}
\label{vertex-degree-five}
Let \(G\) be a connected graph.
\begin{enumerate}
\item If \(\mathrm{Cl}(G)\) is not homotopy equivalent to a wedge of spheres, then \(G\) has a vertex of degree at least 5.
\item If \(G\) is not homotopy \(K\)-invariant, then \(G\) has a vertex of degree at least 5, unless \(G\) is the octahedral graph. \(\qed\)
\end{enumerate}
\end{corollary}

A graph \(G\) such that \(|K^{n}(G)|=1\) for some \(n\) is called \emph{null}. As mentioned in the Introduction, the problem of whether nullity implies contractibility was considered in \cite{contractibility}. Our Main Theorem allows us to prove that this is so for low degree graphs.

\begin{corollary}
\label{ldg-nulls-are-contractible}
If a low degree graph \(G\) is null, then it is contractible.
\end{corollary}

\begin{proof}
Let \(G\) be a null low degree graph. Then \(K^{2}(G)\) is null and Helly. Theorem 2.2 from \cite{BP91} implies that \(K^{2}(G)\) is dismantlable, hence contractible. Since \(G\) is homotopy \(K\)-permanent, \(G\) is also contractible.
\end{proof}

\section{Statement}
\label{sec:org4ab2a9b}

On behalf of all authors, the corresponding author states that there is no conflict of interest.

\bibliographystyle{plain}
\bibliography{homotopy-low}
\end{document}

%% file: intersection-in-one.tikz
\begin{tikzpicture}
    [vertex/.style={circle, fill, draw, inner sep=0pt, minimum size=4pt},
    edge/.style={thin},
    edashed/.style={dashed, thin}]
    \newcommand{\nodecab}[3]{\node at (#1,#2) [vertex] (#3) {};\node at (#3) [above] {\(#3\)};}
    \newcommand{\nodepab}[3]{\node at (#1:#2) [vertex] (#3) {};\node at (#3) [above] {\(#3\)};}
    \newcommand{\nodepbe}[3]{\node at (#1:#2) [vertex] (#3) {};\node at (#3) [below] {\(#3\)};}
    \nodecab{0}{0}{x}
    \nodecab{{sqrt(3)}}{0}{f}
    \nodecab{{-sqrt(3)}}{0}{e}
    \nodepab{45}{1}{c}
    \nodepab{135}{1}{a}
    \nodepbe{-135}{1}{b}
    \nodepbe{-45}{1}{d}
    \node at ({sqrt(2)/3},0) {\(T_{2}\)};
    \node at ({-sqrt(2)/3},0) {\(T_{1}\)};
    \draw[edge] (x) -- (a) -- (e) -- (b) -- (x) -- (c) -- (f) -- (d) -- (x);
    \draw[edge] (a) -- (c) -- (d) -- (b) -- (a);
\end{tikzpicture}

%% file: intersection-in-two.tikz
\begin{tikzpicture}
  [vertex/.style={circle, fill, draw, inner sep=0pt, minimum size=4pt},
  edge/.style={thin},
  edashed/.style={dashed, thin}]
  \newcommand{\nodecab}[3]{\node at (#1,#2) [vertex] (#3) {};\node at (#3) [above] {\(#3\)};}
  \newcommand{\nodepab}[3]{\node at (#1:#2) [vertex] (#3) {};\node at (#3) [above] {\(#3\)};}
  \newcommand{\nodepbe}[3]{\node at (#1:#2) [vertex] (#3) {};\node at (#3) [below] {\(#3\)};}
  \nodecab{0}{0}{x}
  \nodepab{0}{1}{y}
  \nodepab{0}{2}{y'}
  \nodepab{180}{1}{x'}
  \nodepab{60}{1}{u}
  \nodepbe{-60}{1}{u'}
  \node at (30:{sqrt(3)/3}) {\(T_{1}\)};
  \node at (-30:{sqrt(3)/3}) {\(T_{2}\)};
  \draw[edge] (u) -- (x') -- (u') -- (y') -- (u) -- (x) -- (u') -- (y) -- (u);
  \draw[edge] (x') -- (x) -- (y) -- (y');
\end{tikzpicture}

%% file: diagram-hch.tikz
\begin{tikzpicture}
  [scale=1.5,
  vertex/.style={circle, fill, draw, inner sep=0pt, minimum size=4pt},
  edge/.style={thin},
  edashed/.style={dashed, very thin}]
  \newcommand{\nodecab}[3]{\node at (#1,#2) [vertex] (#3) {};\node at (#3) [above] {\(#3\)};}
  \newcommand{\nodecbe}[3]{\node at (#1,#2) [vertex] (#3) {};\node at (#3) [below] {\(#3\)};}
  \newcommand{\nodepab}[3]{\node at (#1:#2) [vertex] (#3) {};\node at (#3) [above] {\(#3\)};}
  \newcommand{\nodepbe}[3]{\node at (#1:#2) [vertex] (#3) {};\node at (#3) [below] {\(#3\)};}
  \node at (-1,0) [vertex] (a) {};
  \node at (0,0) [vertex] (c') {};
  \node at (1,0) [vertex] (b) {};
  \node at (-1/2,{sqrt(3)/2}) [vertex] (b') {};
  \node at (1/2,{sqrt(3)/2}) [vertex] (a') {};
  \node at (0,{sqrt(3)}) [vertex] (c) {};
  \draw[edge] (a) -- (c') -- (b) -- (a') -- (c) -- (b') -- (a);
  \draw[edge] (a') -- (b') -- (c') -- (a');
  \draw[edashed] (a) -- (a');
  \draw[edashed] (b) -- (b');
  \draw[edashed] (c) -- (c');
  \node at (0, {1/sqrt(3)}) [below] {\(q_{1}\)};
  \node at (0, {2/sqrt(3)}) {\(q_{2}\)};
  \node at (-1/2, {1/(2*sqrt(3))}) {\(q_{3}\)};
  \node at (1/2, {1/(2*sqrt(3))}) {\(q_{4}\)};
  \node at (-1/2,{sqrt(3)/2}) [above left=-2pt] {\(x\)};
  \node at (1/2,{sqrt(3)/2}) [above right=-2pt] {\(z\)};
  \nodecbe{1}{0}{x'}
  \nodecbe{0}{0}{y}
  \nodecbe{-1}{0}{z'}
  \nodecab{0}{{sqrt(3)}}{y'}
\end{tikzpicture}

%% file: homotopy-low.bbl
\begin{thebibliography}{10}

\bibitem{2016-adamaszek-adams-frick-peterson-previte-johnson-nerve-complexes-of-circular-arcs}
M.~Adamaszek, H.~Adams, F.~Frick, C.~Peterson, and C.~Previte-Johnson.
\newblock Nerve complexes of circular arcs.
\newblock {\em Discrete Comput. Geom.}, 56(2):251--273, 2016.

\bibitem{BP91}
H.~Bandelt and E.~Prisner.
\newblock Clique graphs and {H}elly graphs.
\newblock {\em J. Combin. Theory Ser. B}, 51(1):34--45, 1991.

\bibitem{2012-barmak-minian-strong-homotopy-types-nerves-and-collapses}
J.~A. Barmak and E.~G. Minian.
\newblock Strong homotopy types, nerves and collapses.
\newblock {\em Discrete Comput. Geom.}, 47(2):301--328, 2012.

\bibitem{1995-bjorner-topological-methods}
A.~Bj{\"{o}}rner.
\newblock Topological methods.
\newblock In {\em Handbook of combinatorics, {V}ol. 1, 2}, pages 1819--1872.
  Elsevier Sci. B. V., Amsterdam, 1995.

\bibitem{2010-boulet-fieux-jouve-simplicial-simple-homotopy-of-flag-complexes-in-terms-of-graphs}
R.~Boulet, E.~Fieux, and B.~Jouve.
\newblock Simplicial simple-homotopy of flag complexes in terms of graphs.
\newblock {\em European J. Combin.}, 31(1):161--176, 2010.

\bibitem{MR0329947}
F.~Escalante.
\newblock \"{U}ber iterierte {C}lique-{G}raphen.
\newblock {\em Abh. Math. Sem. Univ. Hamburg}, 39:59--68, 1973.

\bibitem{MR2059528}
M.~E. Fr{\'i}as-Armenta, V.~Neumann-Lara, and M.~A. Piza\~na.
\newblock Dismantlings and iterated clique graphs.
\newblock {\em Discrete Math.}, 282(1-3):263--265, 2004.

\bibitem{MR0256911-djvu}
F.~Harary.
\newblock {\em Graph theory}.
\newblock Addison-Wesley Publishing Co., Reading, Mass.-Menlo Park,
  Calif.-London, 1969.

\bibitem{MR1867354}
A.~Hatcher.
\newblock {\em Algebraic topology}.
\newblock Cambridge University Press, Cambridge, 2002.

\bibitem{MR3129782}
F.~Larri\'on, M.~A. Piza\~na, and R.~Villarroel-Flores.
\newblock Discrete {M}orse theory and the homotopy type of clique graphs.
\newblock {\em Ann. Comb.}, 17(4):743--754, 2013.

\bibitem{contractibility}
F.~Larri{\'o}n, M.~A. Piza{\~n}a, and R.~Villarroel-Flores.
\newblock Contractibility and the clique graph operator.
\newblock {\em Discrete Math.}, 308(16):3461--3469, 2008.

\bibitem{equivariantc}
F.~Larri{\'o}n, M.~A. Piza{\~n}a, and R.~Villarroel-Flores.
\newblock Equivariant collapses and the homotopy type of iterated clique
  graphs.
\newblock {\em Discrete Math.}, 308(15):3199--3207, 2008.

\bibitem{LPV08a}
F.~Larri{\'o}n, M.~A. Piza{\~n}a, and R.~Villarroel-Flores.
\newblock Posets, clique graphs and their homotopy type.
\newblock {\em European J. Combin.}, 29(1):334--342, 2008.

\bibitem{chessboard-graphs}
F.~Larri{\'o}n, M.~A. Piza{\~n}a, and R.~Villarroel-Flores.
\newblock The clique operator on matching and chessboard graphs.
\newblock {\em Discrete Math.}, 309(1):85--93, 2009.

\bibitem{LARRION2001}
F.~Larrión, V.~Neumann-Lara, and M.~A. Pizaña.
\newblock {On the homotopy type of the clique graph}.
\newblock {\em {Journal of the Brazilian Computer Society}}, 7:69 -- 73, 00
  2001.

\bibitem{zbMATH03641500}
V.~{Neumann-Lara}.
\newblock {On clique-divergent graphs}.
\newblock {Probl\`emes combinatoires et th\'eorie des graphes, Orsay 1976,
  Colloq. int. CNRS No. 260, 313-315 (1978).}, 1978.

\bibitem{Pri92}
E.~Prisner.
\newblock Convergence of iterated clique graphs.
\newblock {\em Discrete Math.}, 103(2):199--207, 1992.

\bibitem{2022-on-the-clique-behavior-of-graphs-of-low-degree}
R.~Villarroel-Flores.
\newblock On the clique behavior of graphs of low degree.
\newblock {\em To appear in Boletín de la Sociedad Matemática Mexicana},
  2022.
\newblock https://arxiv.org/abs/2111.02483.

\end{thebibliography}
